\documentclass[11pt]{article}%
\usepackage{amssymb}
\usepackage{euler}
\usepackage{amsmath}
\usepackage{amsfonts}
\usepackage{graphicx}%
\providecommand{\U}[1]{\protect\rule{.1in}{.1in}}
\newtheorem{theorem}{Theorem}

\newtheorem{proposition}[theorem]{Proposition}

\newenvironment{proof}[1][Proof]{\noindent\textbf{#1.} }{\ \rule{0.5em}{0.5em}}
\begin{document}

\title{\textsf{Closing in on the kernel of an operator between Banach spaces}}
\author{\textsf{Douglas S. Bridges}}
\maketitle

\begin{abstract}%
%TCIMACRO{\TeXButton{noindent}{\noindent}}%
%BeginExpansion
\noindent
%EndExpansion
\textsf{This note deals with the question: If }$T:X\rightarrow Y$\textsf{\ is
a linear mapping between Banach spaces, }$x\in X$\textsf{, and }$\left\Vert
Tx\right\Vert $\textsf{\ is small, is }$x$\textsf{\ close to the kernel of
}$T$\textsf{? It draws on notions of Z-stability, introduced in \cite{BDM2016}%
, and provides an affirmative constructive answer when }$T$\textsf{\ is onto
}$Y$\textsf{, sequentially continuous, and has located kernel.}

\end{abstract}

%

%TCIMACRO{\TeXButton{sf}{\normalfont\sf}}%
%BeginExpansion
\normalfont\sf
%EndExpansion
%

%TCIMACRO{\TeXButton{bigskip}{\bigskip}}%
%BeginExpansion
\bigskip
%EndExpansion
%

%TCIMACRO{\TeXButton{bigskip}{\bigskip}}%
%BeginExpansion
\bigskip
%EndExpansion
%

%TCIMACRO{\TeXButton{noindent}{\noindent}}%
%BeginExpansion
\noindent
%EndExpansion
Working, as we do throughout this note, in Bishop's constructive mathematics,
\textbf{BISH} \cite{Bishop,BVtech,Handbook}, we say that a mapping $f$ of a
metric space $(X,\rho)$ into a normed space $Y$ is

\begin{itemize}
\item[$\bullet$] \textbf{Z-stable at }$x\in X$ if for each $\varepsilon>0$
there exists $\delta>0$ such that if $\left\Vert f(x)\right\Vert <\delta$,
then $\rho(x,z)<\varepsilon$ for some $z\in f^{-1}(0)$;

\item[$\bullet$] \textbf{Z-stable on} $S\subset X$ if it is Z-stable at each
point of $S$; and

\item[$\bullet$] \textbf{uniformly Z-stable on} $S\subset X$ if for each
$\varepsilon>0$ there exists $\delta>0$ such that if $x\in S$ and $\left\Vert
f(x)\right\Vert <\delta$, then $\rho(x,z)<\varepsilon$ for some $z\in
f^{-1}(0)$; and\footnote{%
%TCIMACRO{\TeXButton{sf}{\normalfont\sf}}%
%BeginExpansion
\normalfont\sf
%EndExpansion
For more on Z-stability see \cite{BDM2016}.}

\item[$\bullet$] \textbf{well behaved }if $f(x)\neq0$ whenever\footnote{%
%TCIMACRO{\TeXButton{sf}{\normalfont\sf}}%
%BeginExpansion
\normalfont\sf
%EndExpansion
Recall that the \textbf{complement} (in $X$) of a subset $S$ of $X$ is $%
%TCIMACRO{\TeXButton{mathsim}{\mathord{\sim}}}%
%BeginExpansion
\mathord{\sim}%
%EndExpansion
S\equiv\left\{  x\in X:\forall_{s\in S}(x\neq s)\right\}  $, where the
expression $(x\neq s)$ stands for $(\rho(x,s)>0)$, which is constructively
stronger than $\lnot(x=s)$.} $x\in%
%TCIMACRO{\TeXButton{mathsim}{\mathord{\sim}}}%
%BeginExpansion
\mathord{\sim}%
%EndExpansion
f^{-1}(0)$.
\end{itemize}

The following appear as Propositions 2.1 and 2.2 in \cite{BDM2016}.

\begin{proposition}
\label{apr17p1}If $f$ is well behaved and $f^{-1}(0)$ is both inhabited and
located\footnote{%
%TCIMACRO{\TeXButton{sf}{\normalfont\sf}}%
%BeginExpansion
\normalfont\sf
%EndExpansion
That is, $\rho(x,f^{-1}(0))\equiv\inf\left\{  \rho(x,z):f(z)=0\right\}  $
exists for each $x\in X$.} in $X$, then $f$ is Z-stable.
\end{proposition}

\begin{proposition}
\label{apr17p2}If $X$ is complete and $f$ is Z-stable, then $f$ is well behaved.
\end{proposition}

%

%TCIMACRO{\TeXButton{noindent}{\noindent}}%
%BeginExpansion
\noindent
%EndExpansion
On the other hand, we have:

\begin{proposition}
\label{apr17p3}Every linear mapping of a normed space onto a Banach space is
well behaved \emph{\cite[Theorem 1]{BI1990}}.
\end{proposition}

From Propositions \ref{apr17p1} and \ref{apr17p2} we obtain:

\begin{proposition}
\label{apr22p1}If $T$ is a linear mapping of a normed space $X$ onto a Banach
space $Y$, and $\ker T$, the kernel of $T$, is located in $X$, then $T $ is
Z-stable on $X$.
\end{proposition}

In passing, we note that a linear mapping $T:X\rightarrow Y$ between normed
spaces is $Z$-stable if and only if it is an open mapping: that is, there
exists $r>0$ such that if $x\in X$ and $\left\Vert Tx\right\Vert \leq r$, then
$\left\Vert x\right\Vert \leq1$.%

%TCIMACRO{\TeXButton{medskip}{\medskip}}%
%BeginExpansion
\medskip
%EndExpansion

In the present article we are interested in the question:

\begin{quote}
\emph{If }$T$ \emph{is a nonzero\footnote{%
%TCIMACRO{\TeXButton{sf}{\normalfont\sf}}%
%BeginExpansion
\normalfont\sf
%EndExpansion
That is, there exists $x\in X$ such that $Tx\neq0$.} linear mapping of a
Banach space }$X$\emph{\ onto a Banach space} $Y$\emph{, under what conditions
is }$T$\emph{\ }uniformly\emph{\ Z-stable on }$X$\emph{?}(*)
\end{quote}

%

%TCIMACRO{\TeXButton{noindent}{\noindent}}%
%BeginExpansion
\noindent
%EndExpansion
Our answer will use a proposition whose constructive proof we include for the
sake of completeness of exposition.

\begin{proposition}
\label{apr20p1}Let $T$ be a nonzero linear mapping of a normed space $X$ into
a normed space $Y$. Then the following conditions are equivalent.

\begin{itemize}
\item[\emph{(a)}] There exists $r>0$ such that%
\[
\delta_{r}\equiv\inf\left\{  \left\Vert x\right\Vert :x\in X,\left\Vert
Tx\right\Vert \geq r\right\}
\]
exists and is positive;

\item[\emph{(b)}] $T$ is normable,\footnote{%
%TCIMACRO{\TeXButton{sf}{\normalfont\sf}}%
%BeginExpansion
\normalfont\sf
%EndExpansion
\textsf{Recall that a linear mapping }$T:X\rightarrow Y$\textsf{\ between
normed spaces is \textbf{normable} if its \textbf{operator norm} }$\left\Vert
T\right\Vert \equiv\sup\left\{  \left\Vert Tx\right\Vert :x\in X,\left\Vert
x\right\Vert \leq1\right\}  $\textsf{\ exists.}} with positive norm;

\item[\emph{(c)}] $\delta_{r}$ exists for each $r>0$.
\end{itemize}

%

%TCIMACRO{\TeXButton{noindent}{\noindent}}%
%BeginExpansion
\noindent
%EndExpansion
If any, and therefore all, of these conditions holds, then $\left\Vert
T\right\Vert =r\delta_{r}^{-1}$ for each $r>0$.
\end{proposition}

\begin{proof}
Suppose that $r$ and $\delta_{r}\ $exist as in (i). If $x\in X$ and
$\left\Vert x\right\Vert <1$, then $\left\Vert T(\delta_{r}x)\right\Vert
\not \geq r$, so $\left\Vert Tx\right\Vert =\left\Vert T(\delta_{r}%
x)\right\Vert \delta_{r}^{-1}\leq r\delta_{r}^{-1}$. It readily follows that
$T$ is a bounded linear mapping with bound $r\delta_{r}^{-1}$. To prove that
$r\delta_{r}^{-1}$ is the norm of $T$, let $0<\varepsilon<r\delta_{r}^{-1}$.
Then $0<1-r^{-1}\delta_{r}\varepsilon<1$ and therefore%
\[
\frac{\delta_{r}}{1-r^{-1}\delta_{r}\varepsilon}>\delta_{r}.
\]
Choose, in turn, $t$ such that%
\[
0<t<\frac{\delta_{r}}{1-r^{-1}\delta_{r}\varepsilon}-\delta_{r}%
\]
and $\xi\in X$ such that $\left\Vert T\xi\right\Vert \geq r$ and $\left\Vert
\xi\right\Vert <\delta_{r}+t$. With%
\[
x=\frac{r}{\left(  \delta_{r}+t\right)  \left\Vert T\xi\right\Vert }\xi
\]
we have%
\[
\left\Vert x\right\Vert =\frac{r}{\left\Vert T\xi\right\Vert }\cdot
\frac{\left\Vert \xi\right\Vert }{\delta_{r}+t}<1
\]
and%
\[
\left\Vert Tx\right\Vert =\frac{r}{\delta_{r}+t}>\frac{r(1-r^{-1}\delta
_{r}\varepsilon)}{\delta_{r}}=r\delta_{r}^{-1}-\varepsilon.
\]
Since $\varepsilon>0$ is arbitrary, it follows that%
\[
\left\Vert T^{-1}\right\Vert =\sup\left\{  \left\Vert x\right\Vert :x\in
X,\left\Vert Tx\right\Vert \leq1\right\}
\]
exists and equals $r\delta_{r}^{-1}$.

Now suppose that $T$ is normable and $\left\Vert T\right\Vert >0$. For each
$r>0$, if $\left\Vert Tx\right\Vert \geq r$, then $\left\Vert x\right\Vert
\geq r\left\Vert T\right\Vert ^{-1}$. On the other hand, given $\varepsilon
>0$, pick, in turn, $s$ such that $0<s<\left\Vert T\right\Vert $ and%
\[
\frac{1}{\left\Vert T\right\Vert -s}<\frac{1}{\left\Vert T\right\Vert }%
+r^{-1}\varepsilon
\]
and $\xi\in X$ with $\left\Vert \xi\right\Vert =1$ and $\left\Vert
T\xi\right\Vert >\left\Vert T\right\Vert -s$. Setting%
\[
x=\frac{r}{\left\Vert T\right\Vert -s}\xi,
\]
we have%
\[
\left\Vert Tx\right\Vert =\frac{r\left\Vert T\xi\right\Vert }{\left\Vert
T\right\Vert -s}>r
\]
and%
\[
\left\Vert x\right\Vert =\frac{r}{\left\Vert T\right\Vert -s}<r\left(
\left\Vert T\right\Vert ^{-1}+r^{-1}\varepsilon\right)  =r\left\Vert
T\right\Vert ^{-1}+\varepsilon\text{.}%
\]
Since $r$ and $\varepsilon$ are arbitrary, we conclude that for each $r>0$,
$\delta_{r}$ exists and equals $r\left\Vert T\right\Vert ^{-1}$. Thus (b)
implies (c), which trivially implies (a).
\end{proof}

%

%TCIMACRO{\TeXButton{medskip}{\medskip}}%
%BeginExpansion
\medskip
%EndExpansion

We now have some more definitions applicable to a mapping $f:X\rightarrow Y$
between metric spaces. First, we say that $f$ is

\begin{itemize}
\item[$\bullet$] \emph{one-one} if $x=x^{\prime}$ whenever $x,x^{\prime}\in X$
and $f((x)=f(x^{\prime})$; and

\item[$\bullet$] \emph{injective }if $f(x)\neq f(x^{\prime})$ whenever
$x,x^{\prime}\in X$ and $x\neq x^{\prime}$.
\end{itemize}

%

%TCIMACRO{\TeXButton{noindent}{\noindent}}%
%BeginExpansion
\noindent
%EndExpansion
Constructively, one-one is weaker than injective, but every one-one linear
mapping of a normed space onto a Banach space is injective \cite[Cor.
0]{BI1990}.

On the other hand, $f$ is said to be \emph{sequentially continuous }if for
each sequence $\left(  x_{n}\right)  _{n\geq1}$ converging to $x\in X$, the
sequence $\left(  f(x_{n})\right)  _{n\geq1}$ converges to $f(x)$. In
\cite[Thm. 4]{Ish1992}, Ishihara proved that even when $X$ is complete and
separable, the (classically true) equivalence between sequential and pointwise
continuity is constructively equivalent to a certain boundedness condition,
\textbf{BD}-$\mathbf{N}$, which is now known to be unprovable within BISH
\cite{Lietz}.\footnote{%
%TCIMACRO{\TeXButton{sf}{\normalfont\sf}}%
%BeginExpansion
\normalfont\sf
%EndExpansion
However, \textbf{BD}-$\mathbf{N}$ is provable classically, intuitionistically,
and, as Ishihara has shown \cite[Propn. 4]{Ish1992}, within recursive
constructive mathematics.}

In view of this, it is perhaps not so surprising that sequential continuity
replaces the classical boundedness hypothesis and conclusion in various
constructive versions of Banach's open mapping theorem; see \cite[Cor.
1]{Ish1994}, \cite[Thm 4]{BI1998}, and \cite[6.6.4]{BVtech}.\footnote{%
%TCIMACRO{\TeXButton{sf}{\normalfont\sf}}%
%BeginExpansion
\normalfont\sf
%EndExpansion
The following restricted version of Banach's open mapping theorem is
constructively equivalent to \textbf{BD}-$\mathbf{N}$ and hence not provable
within BISH: \ If $T$ is a nonzero bounded linear mapping of a separable
Hilbert space $H$ into itself such that $T^{\ast}$ exists and $T(H)$ is
complete, then $T$ is an open mapping \cite[Thm 5]{BI1998}.} In the same vein
we have Ishihara's constructive version of \emph{Banach's inverse mapping
theorem}:

\begin{quote}
\emph{If }$T$\emph{\ is a one-one, sequentially continuous, linear mapping of
a separable Banach space }$X$\emph{\ onto a Banach space }$Y$\emph{, then the
linear mapping }$T^{-1}$\emph{\ is sequentially continuous} (\cite[Thm
1]{Ish1994}; see also \cite[Thm 3.4]{IshV}).\footnote{%
%TCIMACRO{\TeXButton{sf}{\normalfont\sf}}%
%BeginExpansion
\normalfont\sf
%EndExpansion
Ishihara \cite[Thm 21]{Ish2001} has shown that the classical form of each of
the following functional-analytic theorems is constructively equivalent to
\textbf{BD-}$\mathbf{N}$ and is therefore not provable within BISH: the open
mapping theorem, Banach's inverse function theorem, the closed graph theorem,
the uniform boundedness theorem, the Banach-Steinhaus theorem, the
Hellinger-Toeplitz theorem.}
\end{quote}

\begin{proposition}
\label{apr21p1}Let $X$ be a separable Banach space, and $T$ a nonzero,
one-one, sequentially continuous, linear mapping of $X$ onto a Banach space
$Y$. For each $r>0$ let%
\[
C_{r}\equiv\left\{  \left\Vert Tx\right\Vert :x\in X,\left\Vert x\right\Vert
\geq r\right\}  .
\]
Then the following conditions are equivalent.

\begin{enumerate}
\item[\emph{(i)}] $\delta_{r}\equiv\ \inf C_{r}$ exists for some $r>0$;

\item[\emph{(ii)}] The linear mapping $T^{-1}$ is normable.
\end{enumerate}

%

%TCIMACRO{\TeXButton{noindent}{\noindent}}%
%BeginExpansion
\noindent
%EndExpansion
If either condition holds, then $\left\Vert T\right\Vert >0$ and for each
$r>0$, $\delta_{r}$ exists and equals $r\left\Vert T\right\Vert ^{-1}$.
\end{proposition}

\begin{proof}
Suppose that $r>0$ and $\delta_{r}$ exists. Noting that%
\[
C_{r}=\left\{  \left\Vert y\right\Vert :y\in Y,\left\Vert T^{-1}y\right\Vert
\geq r\right\}  ,
\]
we see from Proposition \ref{apr20p1} that to prove (ii), it will suffice to
show that $\delta_{r}>0$. To that end, construct an increasing binary sequence
$\left(  \lambda_{n}\right)  _{n\geq1}$ such that%
\begin{align*}
\lambda_{n}=0  &  \Rightarrow\delta_{r}<2^{-n},\\
\lambda_{n}=1  &  \Rightarrow\delta_{r}>2^{-n-1}.
\end{align*}
We may assume that $\lambda_{1}=0$. If $\lambda_{n}=0$, choose $x_{n}\in
C_{r}$ such that $\left\Vert Tx_{n}\right\Vert <2^{-n}$. If $\lambda
_{N}=1-\lambda_{N-1}$, set $x_{n}=x_{N-1}$ for all $n\geq N$. It is routine to
prove that if $m>n$, then $\left\Vert Tx_{m}-Tx_{n}\right\Vert <2^{-n+1}$;
whence $\left(  Tx_{n}\right)  _{n\geq1}$ is a Cauchy sequence in the complete
space $Y$. Thus there exists $x_{\infty}\in X$ such that $Tx_{n}\rightarrow
Tx_{\infty}$. By Banach's inverse mapping theorem, $T^{-1} $ is sequentially
continuous on $Y$, so $x_{n}\rightarrow x_{\infty}$. Then $\left\Vert
x_{\infty}\right\Vert =\lim_{n\rightarrow\infty}\left\Vert x_{n}\right\Vert
\geq r$ and therefore, since $T$ is injective, $Tx_{\infty}\neq0$. Pick a
positive integer $N$ such that $\left\Vert Tx_{\infty}\right\Vert >2^{-N}$.
The sequential continuity of $T$ ensures that there exists $\nu\geq N$ with
$\left\Vert Tx_{\nu}\right\Vert >2^{-N}>2^{-\nu}$. Then $\lambda_{\nu}=1$ and
therefore $\delta_{r}>2^{-\nu-1}>0$. Hence (i) implies (ii).

Conversely, since $T$ is nonzero and injective, $T^{-1}$ is nonzero, so if
(ii) holds, then $\left\Vert T^{-1}\right\Vert >0$; in which case, by
Proposition \ref{apr20p1}, $\delta_{r}=r\left\Vert T\right\Vert ^{-1}$ for all
$r>0$. In particular, (ii) implies (i).
\end{proof}

%

%TCIMACRO{\TeXButton{medskip}{\medskip}}%
%BeginExpansion
\medskip
%EndExpansion

It may be asked why we didn't simply use the following argument to prove that
$\delta_{r}>0$ in the preceding proposition:

\begin{quote}
Suppose that $\delta_{r}=0$; then there exists a sequence $\left(
x_{n}\right)  _{n\geq1}$ in $C_{r}$ such that $Tx_{n}\rightarrow0$. By the
Banach inverse function theorem, $x_{n}\rightarrow0$, so there exists $n$ such
that $\left\Vert x_{n}\right\Vert <r$, a contradiction from which we conclude
that $\lnot(\delta_{r}=0)$ and therefore $\delta_{r}>0$.
\end{quote}

%

%TCIMACRO{\TeXButton{noindent}{\noindent}}%
%BeginExpansion
\noindent
%EndExpansion
The problem with this argument is that passing from $\lnot(a=0)$ to $\left(
a>0\right)  $ requires us to accept \emph{Markov's principle,}

\begin{quote}
\textbf{MP: }For each binary sequence $\left(  a_{n}\right)  _{n\geq1}$ such
that $\lnot\forall_{n}(a_{n}=0)$, there exists $n$ such that $a_{n}=1$.
\end{quote}

%

%TCIMACRO{\TeXButton{noindent}{\noindent}}%
%BeginExpansion
\noindent
%EndExpansion
Although MP is acceptable to the Markov school of recursive constructive
mathematics \cite{Kushner}, it represents an unbounded search, and from the
point of view of BISH is therefore seen as essentially nonconstructive.

Now we have our answer to the question (*) posed towards the start of our
paper. Here, $\rho$ denotes the metric induced by the norm on a normed linear space.

\begin{theorem}
\label{apr21c1}Let $X$ be a separable Banach space, and $T$ a nonzero,
sequentially continuous, linear mapping of $X$ onto a Banach space $Y$ such
that $\ker T$ is located in $X$. Suppose there exists $r>0$ such that%
\[
\delta_{r}\equiv\inf\left\{  \left\Vert Tx\right\Vert :x\in X,\rho(x,\ker
T)\geq r\right\}
\]
exists. Then $T$ is uniformly Z-stable on $X$.
\end{theorem}

\begin{proof}
Since $T$ is sequentially continuous, $\ker T$ is closed in $X$. The quotient
space $X/\ker T$ is therefore a Banach space relative to the quotient norm%
\[
\left\Vert x\right\Vert _{X/\ker T}\equiv\rho(x,\ker T)\ \ \ (x\in X)
\]
\cite[2.3.8]{BVtech}. Regarded as a mapping from $X/\ker T$ onto $Y$, $T$ is
nonzero, injective, sequentially continuous, and linear. Since%
\[
\delta_{r}=\inf\left\{  \left\Vert y\right\Vert :y\in Y,\left\Vert
T^{-1}y\right\Vert _{X/\ker T}\geq r\right\}  ,
\]
it follows from Proposition \ref{apr21p1} that the linear mapping
$T^{-1}:Y\rightarrow X/\ker T$ is normable, with positive norm $r\delta
_{r}^{-1}$. If $x\in X$, $\varepsilon>0$, and $\left\Vert Tx\right\Vert
<r^{-1}\delta_{r}\varepsilon$, then%
\[
\rho(x,\ker T)=\left\Vert x\right\Vert _{X/\ker T}=\left\Vert T^{-1}%
(Tx)\right\Vert _{X/\ker T}\leq\left\Vert T^{-1}\right\Vert \left\Vert
Tx\right\Vert \leq r\delta_{r}^{-1}\left\Vert Tx\right\Vert <\varepsilon.
\]
Thus $T$ is uniformly Z-stable on $X$.
\end{proof}

%

%TCIMACRO{\TeXButton{medskip}{\medskip}}%
%BeginExpansion
\medskip
%EndExpansion

We have two final comments:

\begin{enumerate}
\item If $T$ is a nonzero, one-one, linear mapping of a finite-dimensional
normed space $X$ onto a (perforce finite-dimensional) normed space $Y$, and
$\ker T$ is located in $X\,$, then $X/\ker T$ has positive finite dimension,
and the one-one linear mapping $T^{-1}$ of $Y$ onto $X/\ker T$ is normable
\cite[4.1.8]{BVtech}, with positive norm; whence (see the second-last sentence
of the proof of Theorem \ref{apr21c1}) $T$ is uniformly Z-stable on $X$.

\item The classical\footnote{%
%TCIMACRO{\TeXButton{sf}{\normalfont\sf}}%
%BeginExpansion
\normalfont\sf
%EndExpansion
But \emph{not} the constructive \cite[2.1.18]{BVtech}.} least-upper-bound
principle makes the assumptions of separability, locatedness, and the
existence of $\delta_{r}$ redundant in Corollary \ref{apr20p1} and leads to
the \emph{classical} theorem: Every a bounded linear mapping of a Banach space
onto a Banach space is uniformly Z-stable.
\end{enumerate}

%

%TCIMACRO{\TeXButton{bigskip}{\bigskip}}%
%BeginExpansion
\bigskip
%EndExpansion

\vfill

\begin{flushright}
\texttt{{\small dsbridges 010526}}
\end{flushright}

\label{end}
\end{document}